\documentclass[11pt,reqno]{amsart}
\usepackage{amsthm,amssymb,amsfonts}
\usepackage[nosumlimits]{mathtools}

\usepackage{graphicx}
\usepackage{subcaption}
\usepackage[ruled,vlined]{algorithm2e}
\usepackage{algpseudocode}
\usepackage{tikz-cd}
\usepackage{enumerate}
\usepackage{comment}

\usepackage[margin=1in]{geometry}

\usepackage{hyperref,xcolor}
\hypersetup{
    colorlinks,
    linkcolor={red!50!black},
    citecolor={blue!50!black},
    urlcolor={blue!80!black}
}

\newtheorem{theorem}{Theorem}
\newtheorem{lemma}{Lemma}

\newtheorem{proposition}{Proposition}

\newtheorem{example}{Example}
\newtheorem{remark}{Remark}

\newtheorem*{openproblem}{Open Problem}

\newcommand{\trace}{\mathrm{tr}} 
\newcommand{\im}{\mathrm{im}} 
\newcommand{\rank}{\mathrm{rank}} 
\newcommand{\NN}{\mathbb{N}} 
\newcommand{\ZZ}{\mathbb{Z}} 
\newcommand{\RR}{\mathbb{R}} 
\newcommand{\CC}{\mathbb{C}} 
\newcommand{\QQ}{\mathbb{Q}} 
\newcommand{\yy}{{\sf y}}
\newcommand{\osty}{\displaystyle}
\newcommand*\abs[1]{\lvert#1\rvert} 
\usepackage{graphicx}
\title{Unlabeled Sensing Using Rank-One Moment Matrix Completion}

\author{Hao Liang}
\address{KLMM, Academy of Mathematics and Systems Science, Chinese Academy of Sciences, Beijing 100190, China}
\address{University of Chinese Academy of Sciences}

\email{lianghao2020@amss.ac.cn}

\author{Jingyu Lu}
\address{KLMM, Academy of Mathematics and Systems Science, Chinese Academy of Sciences, Beijing 100190, China}
\address{University of Chinese Academy of Sciences}
\email{lujingyu@amss.ac.cn}

\author{Manolis C. Tsakiris}
\address{KLMM, Academy of Mathematics and Systems Science, Chinese Academy of Sciences, Beijing 100190, China}
\address{University of Chinese Academy of Sciences}
\email{manolis@amss.ac.cn}

\author{Lihong Zhi}
\address{KLMM, Academy of Mathematics and Systems Science, Chinese Academy of Sciences, Beijing 100190, China}
\address{University of Chinese Academy of Sciences}
\email{lzhi@mmrc.iss.ac.cn}

\keywords{unlabeled sensing,
birational morphism,
matrix completion,
determinantal variety,
semidefinite relaxation
}

\begin{document}

\begin{abstract}

We study the unlabeled sensing problem that aims to solve a linear system of equations $A x =\pi(y) $ for an unknown permutation $\pi$.   For a generic matrix $A$ and a generic vector $y$, we construct a system of polynomial equations whose unique solution satisfies $ A\xi^*=\pi(y)$.   In particular,  $\xi^*$ can be recovered by solving the rank-one moment matrix completion problem. We propose symbolic and numeric algorithms to compute the unique solution. Some numerical experiments are conducted to show the efficiency and robustness of the proposed algorithms.

\end{abstract}
\maketitle
\tableofcontents
\newpage

\section{Introduction}\label{sec:introduction}

Given a full-rank matrix $A^*\in\RR^{m\times n}$ and a vector $y^*\in\RR^s$ satisfying $m\geqslant s>n>0$, the {\itshape unlabeled sensing problem} \cite{Unnikrishnan&Haghighatshoar&Vetterli2015, Unnikrishnan&Haghighatshoar&Vetterli2018} 
	asks that for an unknown vector $\xi^*\in\RR^n$, if one only knows the vector $y^*\in\RR^s$ consisting of $s$ shuffled entries of $A^*\xi^*$, whether the vector $\xi^*$ is unique and how to recover it efficiently. This problem emerges from various fields of natural science and engineering, such as biology \cite{Huang&Madan1999, Rose&Mian2014, Abid&Zou2018, Ma&Cai&Li2020}, neuroscience \cite{Nejatbakhsh&Varol2021}, computer vision \cite{David&DeMenthon&Duraiswami&Samet2004, Marques&Stosic&Costeira2009, Tsakiris&Peng2019, Ji&Liang&Xu&Zhang2019, Li&Fujiwara&Okura&Matsushita2023} and communication networks \cite{Narayanan&Shmatikov2008, Keller&JafariSiavoshani&Fragouli&Argyraki&Diggavi2009, Song&Choi&Shi2018}. 

Theorem 1 in \cite{Unnikrishnan&Haghighatshoar&Vetterli2015} asserts that the solution of the unlabeled sensing problem is indeed unique if $s\geqslant2n$ and $A^*$ is generic. For the case $m=s$, 
Song, Choi and Shi \cite{Song&Choi&Shi2018} proposed a new method for recovering the vector $\xi^*$  as follows: let  
	\[ q_i(x)=p_i(A^* x)-p_i(y^*),\]
	where  $x=[x_1,\dots,x_n]$ and  $$p_k(y)=\sum_{i=1}^{m}{y_i^k}\in\RR[y]$$
	is the $k$-th power sum of the variables $y=[y_1,\dots,y_m]$. 
	As $p_k$ is a symmetric polynomial in $y$, its value is independent of the order of the variables $y_1, \ldots, y_m$. Suppose $\pi$ is {\color{black} an element} of the permutation group $\Sigma_{m}$, if $\xi^*$ is a solution of
\begin{equation}\label{problem}
A^* x = \pi(y^*), 
\end{equation}	 then it is a root of the polynomials $q_i(x)$, i.e., 
	\[q_i(\xi^*)=p_i(A^* \xi^*)-p_i(y^*)=p_i(\pi(y^*))-p_i(y^*)=0 \]
	for $i=1, \ldots, m$.

 By Theorem 1 in \cite{Song&Choi&Shi2018} {\color{black} and by \cite{Unnikrishnan&Haghighatshoar&Vetterli2015, Unnikrishnan&Haghighatshoar&Vetterli2018}}, with   $A^*\in\RR^{m\times n}$ a given matrix with i.i.d random entries drawn from an arbitrary continuous probability distribution over $\RR$, if $m \geqslant 2n$, then with probability $1$, $\xi^*$ is the unique solution of the polynomial system 
	\begin{equation}\label{eqn:m}
		Q_m=\{q_1(x)=0, \ldots,~q_m(x)=0\}.
	\end{equation}
Numerical experiments showed that solving the first $n+1$ equations is sufficient for recovering the solution $\xi^*$; i.e. $\xi^*$ is the unique solution of $Q_{n+1}$. Hence, \cite{Song&Choi&Shi2018} pointed out the following open problem {\color{blue}(see also Conjecture 6 in \cite{melanova2022recovery})}:
\begin{openproblem}\label{openproblem}
For a generic matrix $A^*\in\RR^{m\times n}$, $m \geqslant n+1$, and a permutation $y^*$ of a generic vector in the column space of $A^*$, the $(n+1)$-by-$n$ polynomial system
		\begin{equation}\label{eqn:3}
			Q_{n+1}=\{q_1(x)=0, \ldots,~q_n(x)=0,~q_{n+1}(x)=0\}
		\end{equation}
		has a unique solution $\xi^*$, which satisfies $ A^* \xi^* = \pi(y^*)$.
\end{openproblem}

In \cite{Tsakiris2020}, it is shown that for a generic matrix $A^*\in\CC^{m\times n}$ and any vector $y^*$, the $n$-by-$n$ polynomial system
	\begin{equation}\label{eqn:2}
		Q_n=\{q_1(x)=0, \ldots,~q_n(x)=0\}
	\end{equation}
	in the variables $x=x_1,\dots,x_n$ has at most $n!$ solutions. Furthermore, if $y^*$ is a permutation of a generic vector in the column space of $A^*$, then among the solutions of $Q_n$ there is only one vector $\xi^*$ satisfying $A^* \xi^* = \pi(y^*)$. They can recover $\xi^*$ by solving $Q_n$  via symbolic or homotopy methods and then select $\xi^*$ from $n!$ solutions via numerical optimization methods.

\begin{example}\label{running_example_1}
    Given a matrix $A^*$ and a vector $y^*$
    \[A^*:= \left[ \begin {array}{cc} 1&2\\ \noalign{\medskip}4&3
\\ \noalign{\medskip}0&-2\\ \noalign{\medskip}-2&0\end {array}
 \right], ~~ y^*= \left[ \begin {array}{c} -5\\ \noalign{\medskip}-10
\\ \noalign{\medskip}2\\ \noalign{\medskip}4\end {array} \right],\]
find a solution $\xi^*$ such that 
\[A^* \xi^*=\pi(y^*)\]
for an unknown permutation $\pi$ of the coordinates of $y^*$.

From the matrix $A^*$ and the vector   $y^*$, we compute the polynomials 
\begin{eqnarray*}
q_1(x) &=& 3\,x_{{1}}+3\,x_{{2}}+9,\\
q_2(x)&=& 21\,{x_{{1}}}^{2}+28\,x_{{1}}x_{{2}}+17\,{x_{
{2}}}^{2}-145,\\
q_3(x) &=& 57\,{x_{{1}}}^{3}+150\,{x_{{1}}}^{2}x_{{2}}+120\,x_{{1}}
{x_{{2}}}^{2}+27\,{x_{{2}}}^{3}+1053,\\
q_4(x) &=& 16\,{x_{{1}}}^{3}x_{{2}} + 44\,{x_{{1}}}^{2}{x_{{2}}}^{2} + 24\,x_{{1}}{x_{{2}}}^{3}-400.
\end{eqnarray*}

The polynomial system $\{q_1(x)=0,q_2(x)=0,q_3(x)=0, q_4(x)=0\}$
has a unique  solution 
$\xi^*=(-1,-2).$
One can check that $\xi^*$ is the solution of (\ref{problem}) for the permutation $\pi= [1, 2, 4, 3]$. 

We can check that $\xi^*$  is the unique solution of the polynomial system $\{q_1(x)=0,q_2(x)=0,q_3(x)=0\}$, which gives an example supporting the positive answer to the open problem (\ref{openproblem}). 

The polynomial system $\{q_1(x)=0, q_2(x)=0\}$
has two solutions  
$\eta^*=\left(-{\frac{4}{5}},-{\frac{11}{5}}\right), ~~
\xi^*= (-1,-2).$
One can check that  $\eta^*$ is not a solution of (\ref{problem}). 
\end{example}


{ \color{black}
 

   The following theorem shows that one can obtain the unique and desired vector $\xi^*$ for a generic matrix $A^*$ and a permutation $y^*$ of a generic vector in the column space of $A^*$ by solving the polynomial system $Q_{n+1}$. 
  
	\begin{theorem}\label{thm:1}
		The map
		\begin{align}
			f:\CC^{m\times n}\times\CC^n&\rightarrow
   \CC^{m\times n}\times\CC^{n+1}\label{birational:1}\\
			(A^*,\xi^*)&\mapsto(A^*,p_1(A^*\xi^*),\dots, p_{n+1}(A^*\xi^*))\notag
		\end{align}
		is a birational morphism to the image $f(\CC^{m\times n}\times\CC^n)\subseteq\CC^{m\times n}\times\CC^{n+1}$.
  
        \end{theorem}

}

 \begin{remark}
    The fact that the morphism (\ref{birational:1}) is birational implies that there exists a dense open subset $U_2$ of $V$ such that $f$ induces a bijection from $U_1:=f^{-1}(U_2)$ to $U_2$, {\color{black} see \cite[Chapter I, Corollary 4.5]{Hartshorne_1977}.}
		Therefore, if the sample point $(A^*, y^*)$ satisfies the condition $(A^*, p_1(y^*),\dots, p_{n+1}(y^*))\in U_2$, then there exists a unique vector $\xi^*$ such that $(A^*,\xi^*) \in U_1$ and $$(A^*,p_1(A^*\xi^*),\dots, p_{n+1}(A^*\xi^*))=(A^*,p_1(y^*),\dots, p_{n+1}(y^*)).$$
		This answers the open problem above in the affirmative.
 \end{remark}

As an alternative to existing RANSAC \cite{Elhami-ICASSP17}, expectation maximization \cite{Abid-Allerton2018}, branch \& bound \cite{Peng-SPL2020} or homotopy and Groebner bases \cite{Tsakiris2020} approaches for solving the unlabeled sensing problem in its generality (that is without special assumptions such as sparsity, as, e.g., considered in \cite{Slawski-JoS19}), in this paper we focus on finding the unique solution of the polynomial system $Q_{n+1}$ efficiently by reducing it to a rank-one moment matrix completion problem, which can then be solved by many efficient algorithms.   

\paragraph*{\textbf{Structure of the paper}} 

In Section \ref{sec:proof}, we briefly review the main idea in the proof of Theorem \ref{thm:1} in \cite{Liang&Lu&Tsakiris&Zhi2023} and its relationship with previous known results. In Section \ref{sec:moment}, we show that solving polynomial system $Q_{n+1}$ can be reduced to solving the rank-one moment matrix completion problem, which can be solved efficiently by solving a sequence of  {\color{black} semidefinite programming (SDP) problems}.  
In Section \ref{sec:experiment}, we propose some symbolic and numerical algorithms based on  Theorem \ref{thm:1} and \ref{cor}. We conducted eight numeric and symbolic experiments to test the efficiency and robustness of the proposed algorithms. 



\section{Proof of Theorem \ref{thm:1}}\label{sec:proof}

{\color{black}

For convenience, we will abbreviate the unlabeled sensing problem for case $m=s>n>0$ by USP and call the data $(A^*,y^*)$ in USP as the sample point of USP.

The morphism  $f$ can be written as the composition of the following two morphisms 
	\begin{align}
					g:\CC^{m\times n}\times\CC^n&\rightarrow \overline{g(\CC^{m\times n}\times\CC^n)}\subseteq\CC^{m\times n}\times\CC^{m}\label{birational:2}\\
					(A^*,\xi^*)&\mapsto(A^*,p_1(A^*\xi^*),\dots, p_{m}(A^*\xi^*))\notag\\
					\text{and}~\gamma:\overline{g(\CC^{m\times n}\times\CC^n)}&\rightarrow\overline{f(\CC^{m\times n}\times\CC^n)}\subseteq\CC^{m\times n}\times\CC^{m}.\label{gamma} \\
					(A^*,p_1(\eta^*),\dots,p_m(\eta^*))&\mapsto(A^*,p_1(\eta^*),\dots,p_{n+1}(\eta^*))\label{gamma_mapping}
				\end{align}
Since the composition of two birational morphisms is still a birational morphism, the proof of Theorem \ref{thm:1} can be divided into two parts: 	
\begin{itemize}
    \item  Show the morphism $g$ is birational, and

	\item Show the morphism $\gamma$ is birational.
\end{itemize}

   When $m=n+1$,  the morphism $g$ equals the morphism $f$ in Theorem \ref{thm:1}. We show first in Section  \ref{2.1} that the morphism $g$ is birational. Hence,  it implies Theorem \ref{thm:1} for the case $m=n+1$, i.e., one can obtain the unique and desired vector $\xi^*$ for a generic matrix $A^*$ and a permutation $y^*$ of a generic vector in the column space of $A^*$ by solving the polynomial system $Q_{n+1}$ for the case $m=n+1$. Moreover, as $g$ is birational, it implies that for generic sample points $(A^*,y^*)$ of the unlabeled sensing problem (USP), the polynomial system $Q_m$ in (\ref{eqn:m}) has a unique solution $x=\xi^*$, which is also the unique solution of USP.

    Since $m$ can be much larger than $n$ in the application,  we need to show that $f$ is birational for a general $m \geqslant n+1$. This implies that the smaller polynomial system $Q_{n+1}$ in (\ref{eqn:3}) has a unique solution for generic $(A^*,y^*)$. The key point of showing that $f$ is birational is to prove 
    \begin{align*}
        &[\CC(A,x):\CC(A,p_1(Ax),\dots,p_{n}(Ax))]=n!,~\text{and}\\
        &[\CC(A,p_1(Ax),\dots,p_{n+1}(Ax)):\CC(A,p_1(Ax),\dots,p_{n}(Ax))]=n!.
    \end{align*}
    The first degree is the B{\'e}zout number of the complete intersection given by the regular sequence $p_1(Ax),\dots,p_{n}(Ax)$, and the second one is exactly the degree of the minimal polynomial of $p_{n+1}(Ax)$ over $\CC(A,p_1(Ax),\dots,p_{n}(Ax))$. In Section \ref{2.2}, we give a sketch proof showing that this polynomial is the resultant of the system   $  \{p_k(Ax)-z_k:k=1,\dots,n+1\}$.
    
  }

	\subsection{The morphism $g$ is birational }\label{2.1}
			
           {\color{black}
                
            In this subsection, we prove the following theorem, which implies that one can obtain the unique and desired vector $\xi^*$ for a generic matrix $A^*$ and a permutation $y^*$ of a generic vector in the column space of $A^*$ by solving the polynomial system $Q_{m}$.

             \begin{theorem}\label{thm:2} For $m>n>0$,
				the morphism
				\begin{align*}
					g:\CC^{m\times n}\times\CC^n&\rightarrow 
     \CC^{m\times n}\times\CC^{m} 
     \\
					(A^*,\xi^*)&\mapsto(A^*,p_1(A^*\xi^*),\dots, p_{m}(A^*\xi^*))\notag
				\end{align*}
				is birational onto the image $g(\CC^{m\times n}\times\CC^n)$. 		
			\end{theorem}
    When $m=n+1$,  the morphism $g$ equals $f$, i.e., Theorem \ref{thm:2} implies  Theorem \ref{thm:1}  in this special case.}

   We decompose $g$ into the composition of two dominant morphisms $\alpha$ and $\beta$ defined by  (\ref{alpha}) and (\ref{beta}) respectively.  To show $g$ is birational, it suffices to show $\alpha$ and $\beta$ are birational.

   Define the morphism $\alpha$
			\begin{align}
				\alpha:\CC^{m\times n}\times\CC^n&\rightarrow D_{n+1}(A|y)\subseteq\CC^{m\times n}\times\CC^m, \label{alpha}\\
				(A^*,\xi^*)&\mapsto(A^*,A^*\xi^*)\notag
			\end{align}
   where {\color{black} $D_{n+1}(A|y)$ is the zero locus of the all $(n+1)$-minors of the matrix $(A|y)$, $A=[a_{ij}]$ is a $m$-by-$n$ matrix of $mn$ variables and $y=[y_i]$ is a column vector of $m$ variables, ${i\in[m],j\in[n]}$. Namely, $D_{n+1}(A|y)$ is the determinantal variety consisting of $m$-by-$(n+1)$ matrices of rank less than $n+1$. The determinantal variety $D_{n+1}(A|y)$ is irreducible \cite[Proposition 1.1]{Bruns&Vetter1988}, i.e., it is not a union of any two proper closed subvarieties.} 

    		\begin{lemma}\label{lem:1}
    			The morphism $\alpha$ is birational. 
    		\end{lemma}
    		
    		\begin{proof}
    			The determinantal variety $D_{n+1}(A|y)$ is irreducible \cite[(1.1) Proposition]{Bruns&Vetter1988}, and {\color{black}the morphism $\alpha$ maps the non-empty open subset $S_0$ of $\CC^{m\times n}\times\CC^n$ to the non-empty open subset $S_1=\alpha(S_0)$ in $D_{n+1}(A|y)$, where
                \begin{align}
                    S_0&:=\{(A^*,{\color{black}\xi^*})\in\CC^{m\times n}\times\CC^n:\rank(A^*)=n\}\label{solution_set}~\text{and}\\
                    S_1&:=\{(A^*,\eta^*)\in D_{n+1}(A|y):\rank(A^*)=n\} \label{sorted_samples}
                \end{align}
                are respectively the set of the solutions to USP and the set of correctly sorted sample points in USP. Hence, we deduce that $\alpha$ 
                is dominant. Furthermore, for any $(A^*,\eta^*)\in S_1$, using Cramer's rule in linear algebra, we see that the fiber $\alpha^{-1}(A^*,\eta^*)$ is a singleton, hence the morphism $\alpha$ is restricted to a bijection from $S_0$ to $S_1$. Therefore, $\alpha$ is birational.} 
    		\end{proof}
			
	
			The second dominant morphism, to be shown to be birational, is 
			\begin{align}
				\beta:D_{n+1}(A|y)&\rightarrow\overline{\beta(D_{n+1}(A|y))}\subseteq\CC^{m\times n}\times\CC^m \label{beta}\\
                (A^*,\eta^*)&\mapsto(A^*,p_1(\eta^*),\dots,p_m(\eta^*)) \notag
			\end{align}
            Note that {\color{black} $\beta$ is the restriction of the finite morphism 
            \begin{align}
                \CC^{m\times n}\times\CC^m&\rightarrow\CC^{m\times n}\times\CC^m,\label{symmetric_quotient}\\(A^*,\eta^*)&\mapsto(A^*,p_1(\eta^*),\dots,p_m(\eta^*))\notag
            \end{align}
            hence $\beta$ is also a finite and closed morphism \cite[Chapter II, Exercise 3.5]{Hartshorne_1977}, and 
            \begin{equation}
                \beta(D_{n+1}(A|y))=\overline{\beta(D_{n+1}(A|y))}=\overline{g(\CC^{m\times n}\times\CC^n)}.
            \end{equation}}
            {\color{black} To prove the following Lemma \ref{lem:2}, we shall first clarify the actions of $\Sigma_m$ on both $\CC^{m\times n}\times\CC^m$ and $\CC[A,y]$. 
   
            A permutation $\sigma\in\Sigma_m$ acts on $\CC[A,y]$ as a $\CC$-algebra isomorphism given by $\sigma(y_i)=y_{\sigma(i)}$ and fixing $A$. Then 
            for any $f\in\CC[A,y]$ and $y^*=(y_1^*,\dots,y_m^*)\in\CC^m$, we have
            $$(\sigma(f))(y^*)=f\left(y_{\sigma(1)}^*,\dots,y_{\sigma(m)}^*\right).$$
            A permutation $\sigma\in\Sigma_m$ acts on $y^*=(y_1^*,\dots,y_m^*)\in\CC^{m}$ by 
            \begin{equation}
            \sigma(y_1^*,\dots,y_m^*)=\left(y_{\sigma^{-1}(1)}^*,\dots,y_{\sigma^{-1}(m)}^*\right).\label{permute_Cm}
            \end{equation}
            Then 
            $(\sigma(f))(y^*)=f\left(\sigma^{-1}(y^*)\right)$. This action induces the fiberwise action of $\Sigma_m$ on $\CC^{m\times n}\times\CC^m$, i.e., for $(A^*,y^*)\in\CC^{m\times n}\times\CC^m$, we define
            \begin{equation}
                \sigma(A^*,y^*)=(A^*,\sigma(y^*)).\label{permute_fiberwise}
            \end{equation}

            For the variable vector $y=[y_1,\dots,y_m]$, 
            $$\sigma(y):=[y_{\sigma(1)},\dots,y_{\sigma(m)}]$$ 
            is also a vector of variables, and we denote by $D_{n+1}(A|\sigma(y))$ the determinantal variety on which all the $(n+1)$-minors of $(A|\sigma(y))$ vanish. Then we have
            $$D_{n+1}(A|\sigma(y))=\sigma(D_{n+1}(A|y)).$$
            }

	\begin{lemma}\label{lem:2}
		The morphism $\beta$ is birational.
	\end{lemma}

    \begin{proof}
        To show that $\beta$ is birational, it suffices to prove the following three facts
\begin{itemize}
   \item 	The subset
				\begin{equation}
				W_1:=D_{n+1}(A|y)\setminus\bigcup_{\sigma\in\Sigma_m\setminus\{1\}}{D_{n+1}(A|\sigma(y))}\label{W_1}   
				\end{equation}
				is dense open in $D_{n+1}(A|y)$.

				By definition, $W_1$ is an open subset of  $D_{n+1}(A|y)$. According to \cite[Proof of Theorem 1]{Tsakiris2020},  for generic matrix $A^*\in\CC^{m\times n}$, generic column vector $\xi^*\in\CC^n$ and any	$m$-permutation $\sigma\neq1$, we have	
    \[{\color{black}\rank(A^*|\sigma(A^*\xi^*))=n+1.}\]
    Hence $W_1$ is nonempty and dense open in $D_{n+1}(A|y)$.

	\item {\color{black}$\beta^{-1}(\beta(W_1))=W_1$, and the restriction $\beta'$ of $\beta$ on $W_1$ is a bijection from $W_1$ to $\beta(W_1)$.

        We first notice that for any $(A^*,\eta^*)$, all the entries in $\eta^*$ are different from each other. Using Vieta's Theorem, we derive that for any $(A^*,\zeta^*)\in\beta^{-1}(\beta(W_1))$, there is a permutation $\sigma\in\Sigma_m$ such that 
        $$\sigma(\zeta^*)=\eta^*.$$
        Thus $\sigma=1$ and $(A^*,\zeta^*)=(A^*,\eta^*)\in W_1$. Moreover, we also deduce that $\beta'$ is injective.
        
        It is clear that $\beta'$ is surjective, so the bijectivity of $\beta'$ follows.} 

         \item {\color{black} $\beta(W_1)$ is dense open in $\beta(D_{n+1}(A|y))$.

         Vieta's Theorem implies that $\beta$ is a finite morphism, hence $\beta$ is a closed morphism. Since $W_1$ is open in $D_{n+1}(A|y)$ and $\beta^{-1}(\beta(W_1))=W_1$, we conclude that $\beta(W_1)$ is also open in $\beta(D_{n+1}(A|y))$.}

\end{itemize}
\end{proof}

			{\color{black}Now Theorem \ref{thm:2} follows, because $g$ is the composition of the birational morphisms $\alpha$ and $\beta$.}

			\begin{remark}\label{rmk3}
      
     In this remark, we construct explicitly the open subset $W$ of $\mathcal{X}$ such that if the sample point $(A^*,y^*)$ of USP lies in $W$, then there exists a unique $m$-permutation $\pi$ and a unique vector $\xi^*\in\CC^n$ such that $A^*\xi^*=\pi(y^*)$.

    \begin{itemize}
        
        {\color{black}        
        \item The set $S$ of sample points $(A^*,y^*)$ in USP is
    \begin{equation}\label{sample_set}
        S:=\bigcup_{\sigma\in\Sigma_m}{\sigma(S_1)}\subseteq\CC^{m\times n}\times\CC^m, 
    \end{equation}
    where $S_1$ is defined in (\ref{sorted_samples}), and the action of $\Sigma_m$ on $\CC^{m\times n}\times\CC^m$ is defined by (\ref{permute_Cm}) and (\ref{permute_fiberwise}). $S$ is a dense open subset of the $\Sigma_m$-equivariant closed subvariety
    \begin{equation}
        \mathcal{X}:=\bigcup_{\sigma\in\Sigma_m}{\sigma(D_{n+1}(A|y))}\subseteq\CC^{m\times n}\times\CC^m,\label{variety:X}
    \end{equation} 
    A key observation is that, the finite morphism in (\ref{symmetric_quotient}) parameterizes the $\Sigma_m$-orbits of $\mathcal{X}$.
}

    \item  {\color{black} Since $W_1$ in (\ref{W_1}) is nonempty, for any $\tau\in\Sigma_m$ the set
    $$\tau(W_1)=\tau(D_{n+1}(A|y))\setminus\bigcup_{\sigma\in\Sigma_m\setminus\{\tau\}}{\sigma(D_{n+1}(A|y))})$$
    is also nonempty and dense open in $\tau(D_{n+1}(A|y))$. Therefore, the variety $\mathcal{X}$} has exactly $\abs{\Sigma_m}=m!$ irreducible components {\color{black}$\sigma(D_{n+1}(A|y))$ for $\sigma\in\Sigma_m$.} 
    
   		\item 	Write $W_1$ in (\ref{W_1}), then the set
   				$${\color{black} W:=\bigcup_{\sigma\in\Sigma_m}{\sigma (W_1)},\label{W}}$$
   				is dense open in $\mathcal{X}$, hence $S\cap W\neq\varnothing$. Since any point in $W$ lies in a unique irreducible component of $\mathcal{X}$, we conclude that if the sample point $(A^*,y^*)$ of USP lies in $W$, then there exists a unique $m$-permutation $\pi$ and a unique vector $\xi^*\in\CC^n$ such that $A^*\xi^*=\pi(y^*)$. {\color{black}Therefore, the solution $(A^*,\xi^*)$ to USP is unique for most sample points, and in this case $(A^*,\xi^*)$ is exactly the solution of system $Q_m$ in (\ref{eqn:m}).}
   
    \end{itemize}
			\end{remark}

	\subsection{The morphism $\gamma$ is birational }\label{2.2}


    		To show the morphism $\gamma$ in (\ref{gamma}) is birational, we need the projection
				\begin{align}
				\delta:\overline{f(\CC^{m\times n}\times\CC^n)}&\rightarrow\CC^{m\times n}\times\CC^n. \label{deta} \\
					(A^*,p_1(\eta^*),\dots,p_{n+1}(\eta^*))&\mapsto(A^*,p_1(\eta^*),\dots,p_n(\eta^*))\label{delta_mapping}
				\end{align}
			Notice that $(A^*,p_1(\eta^*),\dots,p_{n+1}(\eta^*))$ is a point in $f(\CC^{m\times n}\times\CC^n)$. 
    {Hence, $\delta$ is uniquely determined by (\ref{delta_mapping})}.
    
    {\color{black}Now we give the sketch of the proof that the morphism $\gamma$ is birational.}
    \begin{itemize}
        \item 
  
				For generic $(A^*,w^*)\in\CC^{m\times n}\times\CC^n$, the fiber $(\delta\gamma g)^{-1}(A^*,w^*)$ consists of at most $n!$ points \cite[Theorem 2]{Tsakiris2020} . Since $g$ is birational, we deduce that the fiber $(\delta\gamma)^{-1}(A^*,w^*)$ also generically consists of at most $n!$ points. Hence, it suffices to show that $\delta$ is also a generically finite morphism with generic covering degree $n!$. 
				
	\item			Let $A=[a_{ij}]_{m\times n},t=[t_1,\dots,t_n]$ and $z=[z_1,\dots,z_{n+1}]$ be the matrix and vectors in variables $a_{ij}$, $t_i, z_j$ respectively. 
  {\color{black} Let $R(A,z)\in\ZZ[A,z]$ and $c(A)\in\ZZ[A]$ be the resultants of the polynomial systems
		\begin{gather}
		    \{p_k(At)-z_k:k=1,\dots,n+1\},\label{eqnh:n+1}\\
                \{p_k(At):k=1,\dots,n\}\label{eqnh:n}
		\end{gather}
				eliminating the variables $t_1, \ldots, t_n$, respectively. We shall show that  
    
    \begin{enumerate}
        
        \item\label{property1} The leading term of $R(A,z)$ in $z_{n+1}$ is $c(A)^{n+1}z_{n+1}^{n!}$;

        \item\label{property2} $R(A,p_1(Ax),\dots,p_{n+1}(Ax))=0$;

        \item\label{property3} $R(A,z)$ is irreducible in $\CC(A,z_1,\dots,z_n)[z_{n+1}]$.
    \end{enumerate}
    
    \begin{itemize}
        
        \item To show (\ref{property1}), we use Macaulay's assertions on resultants \cite[Section 8]{Macaulay_1916} to deduce that $\deg_{z_{n+1}}(R(A,z))\leqslant n!$ and the coefficient of $z_{n+1}^{n!}$ in $R(A,z)$ is $c(A)^{n+1}$. According to Lemma 4 in  \cite{Tsakiris2020},  $p_1(A^*t),\dots,p_n(A^*t)$ is a homogeneous regular sequence of $\CC[t]$ for generic matrix $A^*\in\CC^{m\times n}$, hence $p_1(At),\dots,p_n(At)$ is also a homogeneous $\CC(A)[t]$-regular sequence. Thus, the system (\ref{eqnh:n}) in variables $t$ has no solution in the projective $(n-1)$-space over the algebraic closure of $\CC(A)$. Therefore, \cite[Section 10]{Macaulay_1916} asserts that $c(A)$ is nonzero in $\CC[A]$, whence the leading term of $R(A,z)$ in $z_{n+1}$ is $c(A)^{n+1}z_{n+1}^{n!}$.
    
    Moreover, the fact that $p_1(At),\dots,p_n(At)$ is a homogeneous $\CC(A)[t]$-regular sequence implies that the field extension degree $$[\CC(A,x):\CC(A,p_1(Ax),\dots,p_{n}(Ax))]=n!$$
    is the product of the degrees of $p_k(At),k\in[n]$  in $t$. This fact also implies that for generic $(A^*,w^*)\in\CC^{m\times n}\times\CC^n$, the fiber $(\delta\gamma)^{-1}(A^*,w^*)$ consists of $n!$ points.  

    \item To show (\ref{property2}), we notice that the system (\ref{eqnh:n+1}) substituting $z_k$ with $p_k(Ax)$ has a solution $t=x$ in the affine space over the algebraic closure of $\CC(A,p_1(Ax),\dots,p_{n+1}(Ax))$. According to  arguments in  \cite[Section 10]{Macaulay_1916}, we  know that
    $$R(A,p_1(Ax),\dots,p_{n+1}(Ax))=0.$$
    
    \item To show (\ref{property3}), let $\overline{A}$ be the first $n+1$ rows of $A$, and $A_{>n+1}$ be the last $m-(n+1)$ rows of $A$. Then 
    \begin{enumerate}
        \item[(a)] $R\left(\overline{A},z\right)\equiv R(A,z)\mod A_{>n+1}$, and $c\left(\overline{A}\right)\neq0$;

        \item[(b)] The leading term of $R\left(\overline{A},z\right)$ in $z_{n+1}$ is $c\left(\overline{A}\right)^{n+1}z_{n+1}^{n!}$;

        \item[(c)] $R\left(\overline{A},p_1\left(\overline{A}x\right),\dots,p_{n+1}\left(\overline{A}x\right)\right)=0$;

        \item[(d)] $\left[\CC\left(\overline{A},x\right):\CC\left(\overline{A},p_1\left(\overline{A}x\right),\dots,p_n\left(\overline{A}x\right)\right)\right]=n!$.
    \end{enumerate}
         
    Note that the birational morphism $g$ for the case $m=n+1$ implies that $$\CC\left(\overline{A},p_1\left(\overline{A}x\right),\dots,p_{n+1}\left(\overline{A}x\right)\right)=\CC\left(\overline{A},x\right),$$ 
    hence the polynomial $R\left(\overline{A},p_1\left(\overline{A}x\right),\dots,p_n\left(\overline{A}x\right),z_{n+1}\right)$ in $z_{n+1}$ is irreducible over $\CC\left(\overline{A},p_1\left(\overline{A}x\right),\dots,p_n\left(\overline{A}x\right)\right)$. Since $p_1\left(\overline{A}x\right),\dots,p_n\left(\overline{A}x\right)$ is $\CC\left(\overline{A}\right)[x]$-regular, $R\left(\overline{A},z\right)$ in $z_{n+1}$ is also irreducible over $\CC\left(\overline{A},z_1,\dots,z_n\right)$. 
    Using Gauss' Lemma, we see that any factorization $R(A,z)=uv$ in $\CC(A,z_1,\dots,z_n)[z_{n+1}]$ induces a factorization $R(A,z)=u_1v_1$ in $\CC[A,z]$ satisfying $$\deg_{z_{n+1}}u_1=\deg_{z_{n+1}}u,~\deg_{z_{n+1}}v_1=\deg_{z_{n+1}}v$$ 
    and the leading coefficients of $u_1$ and $v_1$ in $z_{n+1}$ ly in $\CC[A]$. Modulo $A_{>n+1}$ and noticing that $c\left(\overline{A}\right)^{n+1}\neq0$ is the leading coefficient of $R\left(\overline{A},z\right)$, we deduce that $R\left(\overline{A},z\right)=\overline{u_1}\overline{v_1}$, $\deg_{z_{n+1}}\overline{u_1}=\deg_{z_{n+1}}u_1$ and $\deg_{z_{n+1}}\overline{v_1}=\deg_{z_{n+1}}v_1$. Thus, $R(A,z)$ is irreducible over $\CC(A,z_1,\dots,z_n)$.
    \end{itemize}
}
    
    Hence, {\color{black}the algebraic field extension degree $$[\CC(A,p_1(Ax),\dots,p_{n+1}(Ax)):\CC(A,p_1(Ax),\dots,p_{n}(Ax))]=n!$$ 
    coincides with 
    $$[\CC(A,x):\CC(A,p_1(Ax),\dots,p_{n}(Ax))]=n!,$$
    which implies that} $\delta$ is a generically finite morphism with a generic covering degree $n!$, {\color{black} and
    $$\CC(A,p_1(Ax),\dots,p_{n+1}(Ax))=\CC(A,x),$$ 
    namely,} $\gamma$ is birational, and  Theorem \ref{thm:1} follows. 
           We refer to \cite{Liang&Lu&Tsakiris&Zhi2023} for detailed proof of the above arguments. 
   
      \end{itemize}	
        
      \begin{remark}
					From the fact that $\gamma$ is birational, we can also derive that for generic $(A^*,w^*)\in\CC^{m\times n}\times\CC^n$, the fiber $(\delta\gamma g)^{-1}(A^*,w^*)$ exactly consists of $n!$ points, among which there is, however, only one solution to USP. This result indicates that the bound $n!$ in \cite[Theorem 2]{Tsakiris2020} is optimal.
				\end{remark}

\section{Rank-one moment matrix completion}\label{sec:moment}

There are many well-known symbolic or numeric algorithms and software packages for solving zero-dimensional polynomial systems in variables $ x=[x_1, \ldots,x_n]$
\begin{equation}\label{ps}
 \{g_1(x)=0, \ldots, g_{m}(x)=0\},  
\end{equation}
e.g. \cite{F4, berthomieu2021msolve, verschelde1999algorithm,Bertini,LLT08,reid2009solving}.   A desirable method for us is the semidefinite relaxation method proposed by Parrilo \cite{parrilo2000structured,parrilo2003semidefinite}, Lasserre \cite{lasserre2001global}
for solving a  sequence of  SDP problems with constant objective function $1$
  \begin{equation}\label{sdp0}
  \left\{
   \begin{array}{cl}
 \osty \min  & 1 \\
 \text{s.\ t.}  &  \yy_0 = 1,\\
                  & M_t(\sf{y})\succeq 0,\\
                  &  M_{t-d_j}(g_j\yy)= 0, \quad {j=1,\ldots,m,}\\
\end{array}\right.\quad
\end{equation}
where $d_j:=\lceil \deg(g_j)/2\rceil$,  $\yy={(\yy_{\alpha})}_{\alpha \in \NN^n}\in \RR^ {\NN^n}$, and  $M_t(\yy)$ {\color{black} { is the truncated moment matrix of order $t$}
\[M_t(\yy):=(\yy_{\alpha+\beta})_{|\alpha|,|\beta|\leqslant t }\in \RR^ {\NN^n\times \NN^n},\]
$|\alpha|=\sum_i \alpha_i$, $|\beta|=\sum_i \beta_i, $  which is the 
 principal submatrix of the full (infinite) moment matrix 
\[M(\yy)=(\yy_{\alpha+\beta}) \in \RR^ {\NN^n\times \NN^n}, ~ \alpha, \beta  \in \NN^n.\]
}
Suppose $g_j(x)=\Sigma_{\alpha\in\NN^n}g_{j,\alpha}
x^{\alpha}\in\RR[x]$. If the $(k,l)$-th entry of  $M_{t}(\yy)$ is
$y_{\beta}$, then the $(k,l)$-th entry of the 
localizing matrix   $M_{t-d_j}(g_j\yy)$ with respect to $\yy$ and $g_j$ is defined by
\begin{align*}
    M_{t-d_j}(g_j\yy)(k,l):=\sum_{\alpha}{g_{j,\alpha}\yy_{\alpha+\beta}}.
\end{align*}

Lasserre, Laurent, and Rostalski gave explicitly rank conditions that guarantee us to find all real solutions of the polynomial system (\ref{ps})  \cite{LLR08, LLR09a, LLR09b} by the semidefinite relaxation method. 
Let 
$d=\max_{j=1, \ldots, m} d_j,$
for $t \geqslant d$,
define the convex set 
\begin{equation}\label{kt}
    K_t^\RR:=\left\{\yy \in \RR^{\NN_{2t}^n} | \yy_0=1, M_t(\yy)\succeq 0, M_{t-d_j}(g_j\yy)=0, j=1,\ldots, m\right\}
\end{equation}
\begin{equation}\label{k}
    K^\RR:=\left\{\yy \in \RR^{\NN^n} | \yy_0=1, M(\yy)\succeq 0, M(g_j\yy)=0, j=1,\ldots, m\right\}
\end{equation}
We introduce the following notations: $T^n:=\{x^\alpha|\alpha\in \NN^n\}$, $T^n_{t}:=\{x^\alpha \in T^n||\alpha|\leq t\}$.
\begin{proposition}\label{prop:sol is opt}
  Let   $I = \langle g_1, \dots, g_m\rangle$ be an ideal in  $\RR[x], V_\RR(I)\neq \O$. If $v \in \RR^n$ is a  real  root of $I$, i.e. $v \in V_\RR(I)$, then for  $t\geq d$, $$\yy^{\delta_v}_{t}=(v^{\alpha })_{\alpha\in T^{n}_{2t}}$$
    is an optimal solution for the optimization problem:
      \begin{equation}\label{sdp*}
  \left\{
   \begin{array}{cl}
 \osty \min  & \rank(M_t(\yy)) \\
 \text{s.\ t.}  &  \yy_0 = 1,\\
                  & M_t(\sf{y})\succeq 0,\\
                  &  M_{t-d_j}(g_j\yy)= 0, \quad {j=1,\ldots,m,}\\
\end{array}\right.\quad
\end{equation}
    
\end{proposition}
\begin{proof}
    Since $\yy^{\delta_v}_{t}$ comes from a Dirac measure $\delta_v$, by definition $(\yy^{\delta_v}_{t})_{0} = v^0 =1$ and $$M_t(\yy^{\delta_v}_{t})=\zeta_{v,t}\zeta_{v,t}^{\top}\succeq 0,$$ where $\zeta_{v,t}=(v^\alpha)_{\alpha\in T_t^{n}}$ is a column vector. For $g_j(x)=\sum_{\alpha \in T^n} g_{j,\alpha}x^\alpha$, define $vec(g_j) = (g_{j,\alpha})_{\alpha}$ Therefore $vec(g_j)^{\top}M_t(\yy^{\delta_v}_{t})vec(g_j)=g_j(v)^2 = 0$, combining $M_t(\yy^{\delta_v}_{t})\succeq 0$, it follows that
    \begin{equation}
        M_t(\yy^{\delta_v}_{t})vec(g_j) = 0, vec(g_j) \in \ker M_t(\yy^{\delta_v}_{t}),M_{t-d_j}(g_j\yy)=0.
    \end{equation}
    We have shown that $\yy^{\delta_v}_{t}\in K_t^\RR$ is a feasible solution of (\ref{sdp*}) and $\rank M_t(\yy^{\delta_v}_{t}) =1$. For 
    $\forall \yy \in K_t^\RR, \yy_0=1\Rightarrow \rank(M_t(\yy))\geq 1$, so the optimal value of (\ref{sdp*}) is $1$, which can be achieved at $\yy^{\delta_v}_{t}$.
\end{proof}

\begin{proposition}
   Let $I = \langle g_1, \dots, g_m\rangle$ be an ideal in $\RR[x], V_\RR(I)\neq \O$. For $t\geq d$, suppose $\hat{\yy}_t\in K_t^\RR$ is an optimal solution of the optimization problem (\ref{sdp*}).
    Then there  $\exists v \in V_\RR(I)$, s.t. $\hat{\yy}_t = \yy^{\delta_v}_{t}$.
\end{proposition}
\begin{proof}
    Because $V_\RR(I) \neq \O$, choose some $\xi_0\in V_\RR(I)$ and construct $y^{\delta_{\xi_0}}_{t}$. According to Proposition \ref{prop:sol is opt}, the optimal value of (\ref{sdp*}) is $1$ and can be achieved at $y^{\delta_{\xi_0}}_{t}$.
    Since $\hat{\yy}_t\in K_t^\RR$ is an optimal solution of (\ref{sdp*}), $\rank(M_t(\hat{\yy}_t))=1$. Notice that $M_0(\hat{\yy}_t)=(1)$, $$1=\rank(M_0(\hat{\yy}_t))\leq\rank(M_s(\hat{\yy}_t))\leq\rank(M_t(\hat{\yy}_t))=1, 0\leq s\leq t. $$  
    $M_t(\hat{\yy}_t)$ satisfies flat extension condition.  Hence $\hat{\yy}_t$ can be uniquely extended to $\tilde{\yy}\in K^\RR$, s.t. $\rank(M(\tilde{\yy}))=\rank(M_t(\hat{\yy}_t))=1$.
    Then Theorem 3.3 in \cite{LLR08} shows that there exists $v \in \RR^{n}$, $\yy = \yy^{\delta_v}=(v^\alpha)_{\alpha\in T^n}$, with $\ker(M(\tilde{\yy}))= I(\{v\})$. Then $M(g_j\tilde{\yy})=0\Rightarrow g_j\tilde{\yy}=M(\tilde{y})vec(g_j)=0$, i.e. $vec(g_j)\in \ker(M(\tilde{\yy}))$, so $g_j(v) =0, j=1,\dots m$, which implies that $v\in V_\RR(I)$.
    
\end{proof}
Combining the above results, we have proved that:
\begin{theorem}
    Suppose $V_\RR(I)\neq \O, t\geq d$ then the map 
    \begin{align*}
        V_\RR(I) &\to Opt(I):=\{\yy \in K_t^\RR~|~\yy \text{ is an optimal solution of (\ref{sdp*})}\}\\
        v &\mapsto \yy^{\delta_v}_t
    \end{align*}
    gives an $1:1$ correspondence between a real root in $V_\RR(I)$ and an optimal solution of (\ref{sdp*})(which implies that the rank of the moment matrix is $1$).
\end{theorem}
According to the above Theorem, to obtain the unique solution of $I$, one must find the truncated moment sequence $\yy\in K_d^{\RR}$ with rank $1$. 
Hence, finding the unique real solution of a polynomial system is reduced to solving the rank-one moment matrix completion problem. There are many well-known methods for solving low-rank matrix completion problems 
\cite{candes2012exact,5454406,5074571,recht2010guaranteed,4797640,vandereycken2013low,jain2013low,8759045, 7383723,cosse2021stable, henrion2020real, nie2014semidefinite,nie2023low, da2016finite}. 

The nuclear norm of a matrix is the sum of the singular values of the matrix. If the matrix is symmetric and positive semidefinite, then its nuclear norm is equal to the trace of the matrix. 
In \cite[Theorem 2.2]{recht2010guaranteed}, Recht, Fazel, and Parrilo have shown that the nuclear norm is the best convex lower approximation of the rank function over the set of matrices with spectral norm less than or equal to $1$. Therefore, the problem of finding  a rank-one moment matrix can be relaxed to  the following form: 
\begin{equation}\label{SDP}
\left\{
   \begin{array}{cl}
 \osty \min \ & \trace(M_t(\yy)) \\
 \text{s.\ t.} \ &  \yy_0 = 1,\\
                  & M_t(\yy)\succeq 0,\\
                  &  M_{t-d_j}(g_j\yy)= 0, \quad {j=1,\ldots,m}\\
\end{array}\right.
\end{equation}

In  \cite{7383723,cosse2021stable}, Cossc and Demanet also show that when the rank-one matrix completion problem has a unique solution,  the second-order semidefinite relaxation with minimization of the nuclear norm will be enough for finding the unique solution. 
In \cite{ma2012computing,yang2023verifyrealroots,yang2013verified}, we have shown that by solving the nuclear norm optimization problem (\ref{SDP}), one can find some real solutions (not all solutions) for polynomial systems more efficiently.

 According to Theorem \ref{thm:1}, the polynomial system $Q_{n+1}$ has a unique real solution.   Hence, in the next section,  instead of minimizing the constant objective function $1$ in (\ref{sdp0}), we solve the nuclear norm minimization problem (\ref{SDP}) to find the unique real solution of $Q_{n+1}$.

\section{Experiments}\label{sec:experiment}
\subsection{Algorithm Design}\label{design}
For simplicity, we assume that our data $A^{*}\in \RR^{m \times n}, y^{*} \in \RR ^{n}$ are generic in the sense of Theorem \ref{thm:1} {and \cite[Theorem 2]{Tsakiris2020}}.
This section presents two symbolic algorithms and one symbolic-numeric algorithm to solve the unlabeled sensing problem (\ref{problem}). 

A straightforward attempt to recover $\xi^{*}$ is to solve the polynomial system $Q_n$ using Groebner basis \cite{GB65,GB70,GB79}, which is formalized as the following algorithm.   The experiment \ref{exp:1} was conducted on Maple's symbolic computation software, and Maple automatically chose the monomial order, usually the reverse-graded lexicographic order.

\begin{algorithm}
\caption{Groebner basis for solving square system $Q_{n}$ (\ref{eqn:2})}\label{alg:symb_n}
\SetAlgoLined

\KwIn{$A^{*}\in \QQ^{m \times n}, y^{*} \in \QQ ^{n}$ such that $\exists \pi \in \Sigma_{m}$, the permutation group, and $\pi( y^*) \in \im(A^*)$, the image space of $A^{*}$}
\KwOut{$\xi^{*}$ such that $A^{*}\xi^{*} = \pi( y^{*})$ }
{compute Groebner basis of the ideal $I_n=(Q_n)$, and extracted the roots symbolically\newline
$roots \gets \{\xi\in \QQ^{n} \vert q_{i}(\xi) = 0, i=1,\dots, n\}$\;}
\For{$\xi \in roots$}
{\If{ coordinates of $A^{*}\xi $ are a permutation of $y^{*}$}
{\Return $\xi^{*} \gets \xi$}
}

\end{algorithm}
The complexity of solving zero-dimensional polynomial systems in $n$ variables is known to be single exponential in $n$ \cite{caniglia1988some, lakshman1991single, F4}, even for well-behaved cases \cite{FMV2016}. The complexity of examining whether coordinates of $A^{*}\xi $ are a permutation of $y^{*}$  is quasi-linear in $m$ using a quick sorting algorithm.

 According to Theorem \ref{thm:1},  the polynomial system $Q_{n+1}$ has only one solution. Experimental results in Section \ref{eval} show that it is far more efficient to compute the Groebner basis for polynomial systems with only one solution. This observation leads to the following algorithm \ref{alg:symb_extra}.

\begin{algorithm}
\caption{Groebner basis for solving overdetermined system $Q_{n+1}$ (\ref{eqn:3})}\label{alg:symb_extra}
\SetAlgoLined
\KwIn{$A^{*}\in \QQ^{m \times n}, y^{*} \in \QQ ^{n}$ such that $\exists \pi \in \Sigma_{m}$, the permutation group, and $\pi( y^*) \in \im(A^*)$, the image space of $A^{*}$}
\KwOut{$\xi^{*}$ such that $A^{*}\xi^{*} = \pi( y^{*})$ }
{compute Groebner basis of the ideal $I_{n+1} = (Q_{n+1})$, and extracted the roots symbolically\newline $roots \gets \{\xi\in \QQ^{n} \vert q_{i}(\xi) = 0, i=1,\dots, n+1\}$\;}
{$roots = \{\xi^{*}\}$\;}
{\Return $\xi^{*}$}
    
\end{algorithm}

Both algorithms above use symbolic computation. Given precise input and clean measurement, the solution one obtains using these two algorithms is exact. When the given data are corrupted by noise,  the perturbed polynomial system  $Q_{n+1}$ will have no solution as it is an overdetermined polynomial system.  However, in real-world applications, all measurements are inevitably influenced by noise. Therefore, a robust and efficient symbolic-numerical algorithm is needed to deal with corrupted measurements $y^{*}$ one observes in the unlabeled sensing problem.  In Section \ref{sec:moment}, we have shown that the unique solution of the polynomial system $Q_{n+1}$ can be obtained by solving the rank-one moment matrix completion problem (\ref{SDP}). Below, we develop the rank-one moment matrix completion method for solving polynomial system $Q_{n+1} $ (\ref{eqn:3}) numerically.

\begin{algorithm}
    \caption{Matrix completion solving overdetermined system $Q_{n+1}$ (\ref{eqn:3})}\label{alg:num}
\KwIn{$A^{*}\in \RR^{m \times n}, y^{*} \in \RR ^{n}$ such that $\exists \pi \in \Sigma_{m},\pi( y^{*}) \in im(A^{*})$}
\KwOut{$\xi^{*}$ such that $A^{*}\xi^{*} = \pi( y^{*})$ }
{\color{black} {solve SDP problem  (\ref{SDP}) for $g_i=q_i$ to the relaxation order $s=\lceil (n+1)/2\rceil$ to get the moment matrix  $M_s$

{compute singular value decomposition  of   $M_s$} 
{$[U,S,V]\gets svd(M_s)$\;}
{select the first n+1 rows of $U$ \newline}
{$x_1\gets U(1:n+1,1)$\;}
{normalize $x_1$ to recover $\xi_{sdp}$}
{$\xi_{sdp} \gets (\xi(k) \gets x_1(k+1,1)/x_1(1,1))_{k = 1,\dots, n}$\;}
{use EM method in \cite{Tsakiris2023} to refine $\xi_{sdp}$ \newline}
{$\xi_{EM} \gets EM(\xi_{sdp})$\;}
\Return $\xi^{*} \gets \xi_{EM}$}}

\end{algorithm}

Semidefinite programming problems can be solved using the interior-point method~\cite{boyd04convex}, whose complexity is bounded by a polynomial in the matrix size. There are well-developed software packages to solve SDP problems, for example, Sedumi~\cite{Sedumi99} and SDPNAL+~\cite{sun2019sdpnal}. 

The  Expectation-Maximization (EM) method proposed in \cite{Tsakiris2020} is an algebraic trick to refine the solution extracted from a numerical polynomial system solver. By sorting $A^{*}\xi_{sdp}$ and $y^*$, if the solution extracted is accurate enough, one expects that $A^{*}\xi_{sdp}$ and $y^*$ after sorting should match, hence recovers the permutation $\pi$ and reduces the problem to classical linear regression. The numerical precision of classical linear regression is determined by the condition number of $A^*$ and precision of $y^*$, which, in practice, can be better controlled in numerical algorithms.

\subsection{Symbolic and Numerical Experimental Results}\label{eval}
The experiments were conducted on a desktop with an Intel(R) Core(TM) i9-10900X CPU @ 3.70GHz and 128 GB RAM. The matrix entries of $A^{*}$ were randomly generated from standard Gaussian distribution in these experiments, and entries $\xi^{*}$ were generated from a uniform distribution in $(0,1)$. Equations $(q_{i}=0)_{i =1,\dots,n+1}$ were then constructed from $A^{*}$, $\xi^{*}$, and possible noise. SDP in algorithm \ref{alg:num} was solved using Sedumi~\cite{Sedumi99}. {\color{black}Source codes are uploaded on GitHub at \href{https://github.com/lujingyv/unlabeled-sensing-ISSAC-2024}{unlabeled-sensing-ISSAC-2024}.}

\subsubsection{symbolic algorithm \ref{alg:symb_n} and \ref{alg:symb_extra}, clean input data in $\QQ$, algorithm performance on different $n$}\label{exp:1}\hfill

Experiment \ref{exp:1} compares the running time of algorithm \ref{alg:symb_n} (TA1/s) and algorithm \ref{alg:symb_extra} (TA2/s) with different $n = 3,4\dots,8$, and $m$ was set as $2n$, clean measurement. Algorithm \ref{alg:symb_n} uses the square system $Q_n=\{q_1(x)=0, \ldots,q_{n}(x)=0\}$. Algorithm \ref{alg:symb_extra} uses the overdeterimined system $Q_{n+1}=\{q_1(x)=0, \ldots,q_n(x),q_{n+1}(x)=0\} $. All of the Groebner basis computations are done in Maple.

\begin{table}[!ht]
    \centering
    \caption{cpu time for Algorithm \ref{alg:symb_n} and Algorithm \ref{alg:symb_extra}, with different $n$, $m=2n$.}
    \begin{tabular}{|l|l|l|}
    \hline
        $n$ & $TA1/$s & $TA2/$s  \\ \hline
        3 & 0.05  & 0.08   \\ \hline
        4 & 1.95  & 0.24   \\ \hline
        5 & 153.05  & 0.56   \\ \hline
        6 & -- & 3.30   \\ \hline
        7 & -- & 200.39   \\ \hline
        8  & -- & 72549.63\\ \hline
    \end{tabular}
    \label{tab:1}
\end{table}

Table \ref{tab:1} shows that as the number of variables and the degree of the polynomial system increase, computing a Groebner basis for the polynomial system $Q_{n}$ or  $Q_{n+1}$ becomes more difficult. As $Q_{n+1}$ has only one solution, its Groebner basis in lexicographic order consists of $n$ linear polynomials $x_1-\xi^*_1, \ldots, x_n-\xi^*_n$. On the other hand, the polynomial system $Q_n$ has at most $n!$ solutions; its Groebner basis consists of a set of polynomials with higher degrees and larger coefficients. Therefore, solving the overdetermined polynomial system $Q_{n+1}$ is much faster than solving the square polynomial system $Q_n$. When $n\geqslant 6$, algorithm \ref{alg:symb_n} could not return the result within 24 hours and was terminated manually.

\subsubsection{numerical algorithm \ref{alg:num}, clean input data in $\RR$, algorithm performance on different $n$}\label{exp:2}\hfill

The following experiments examine the performance of numerical algorithm \ref{alg:num}. Since algorithm \ref{alg:num} computes the moment optimization problem numerically, the relative error of the solution extracted from the moment matrix is defined as
$ err_{sdp} := \frac{\lVert \xi^{*}-\xi_{sdp}\rVert_{2}}{\lVert \xi^{*}\rVert_{2}},$
where $\xi_{sdp}$ represents the solution extracted by semidefinite programming part of algorithm \ref{alg:num}.

Experiment \ref{exp:2} examines the running time of algorithm \ref{alg:num} (TA3/s) with different $n =3,4,5,6$, and $m$ was set as $2n$, $t$ records the relaxation order used in the semidefinite programming, $ranks$ records the rank sequence of the truncated moment matrix $M$, $size$ denotes the dimension of the truncated moment matrix for the given relaxation order. In this experiment, no noise was imposed on $y$.

\begin{table}[!ht]
    \centering
    \caption{CPU time, relative error of $\xi_{sdp}$, rank sequence and moment matrix size for Algorithm \ref{alg:num} with different $n$, $m =2n$.}
    \begin{tabular}{|l|l|l|l|l|l|}
    \hline
        $n$ & $t$ & $TA3/$s & $ err_{sdp}$ & $ranks$ & $size$ \\ \hline
        $3$ & $2$  & $0.20$  &  2.73E-08 &$1,1,1$ & $10$\\ \hline
        $4$ & $3$  & $1.88$  & 6.53E-07  &$1,1,1,1$ & $35$\\ \hline
        $5$ & $3$  & $5.58$  & 1.62E-06  &$1,1,1,1$ & $56$\\ \hline
        $6$ & $4$  & $290.95$  & 4.93E-06  &$1,1,1,1,1$ & $210$\\ \hline
        
    \end{tabular}
    \label{tab:2}
\end{table}

Table \ref{tab:2} shows that as the number of variables and the degree of polynomial rise, the time consumed to compute the moment matrix increases. The rank sequence of the moment matrix $M$ stabilizing at $1$ indicates that the semidefinite programming reaches the flat extension criterion at the lowest relaxation order required, which is $\lceil (n+1)/2\rceil$, therefore Theorem \ref{cor} guarantees that the solution extracted from the $M$ is exactly the unique solution of the polynomial system $Q_{n+1}$. When using clean measurement $y$, the $ err_{sdp}$ behaves well and mildly increases as $n$ becomes larger for $n =3,4,5,6$.

\subsubsection{numerical algorithm \ref{alg:num}, corrupted measurement $y$, algorithm performance on different $m$}\label{exp:3}
\hfill

The following experiments included noise on measurement. For real solution $\xi^{*}$, and clean 
 measurement $y^{*} := A^{*}\xi^{*}$, noise $y_{c}$ was assumed to be a Gaussian random vector with expectation $0$ and covariance matrix $\sigma^{2}I$, where $I$ denotes the $n\times n$ identity matrix. The magnitude of noise in signal processing is measured using signal-noise-ratio $SNR$ in decibels. In our experiments, $SNR$ can be computed from the formula
 {\color{black}\[SNR = 10\left(\log_{10}\left(\frac{n}{3\sigma^{2}}\right)\right).\]}
 
 The equations generated using corrupted measurement $y:=y^{*}+y_{c}$ as
 \[q_{ic}:= p_{i}(A^{*}x) - p_{i}(y)=0.\]
 We further report the relative error of the solution refined by EM algorithm~\cite{Tsakiris2020} as
$ err_{EM} := \frac{\lVert \xi^{*}-\xi_{EM}\rVert_{2}}{\lVert \xi^{*}\rVert_{2}}.$
{\color{black} EM method is essentially a sorting procedure. It sorts $y$ according to its coordinates to $y_{sort}$.  For possible solutions $(x_i)_i$, compute $Ax_i$ and sort its coordinations to $(Ax_i)_{sort}$. The most likely solution should be the nearest one to $y_{sort}$. Hence, the permutation is determined accordingly. Once the permutation is found, solving USP is reduced to solving  a linear system.}

Experiment \ref{exp:3} examines the running time and relative error of algorithm \ref{alg:num} with respect to $m$ under the noise of $SNR = 60$dB, fixing $n=4$.$T_{EM}$, $T_{sdp}$ record the CPU time of EM refinement and solving SDP problem in algorithm \ref{alg:num}, $TA3=T_{EM}+T_{sdp}$ is the total CPU time of algorithm \ref{alg:num}.

\begin{table}[!ht]
    \centering
    \caption{CPU time, and relative error of $\xi_{EM}$ and $\xi_{sdp}$ for algorithm 3 with different m, observation y is corrupted by Gaussian random noise of $SNR = 60$dB, $n =4$ fixed, $20$ trials median.}
     \begin{tabular}{|l|l|l|l|l|l|}
    \hline
        $m$ & $err_{EM}$ & $err_{sdp}$ & $T_{EM}/$s & $T_{sdp}/$s & $TA3/s$  \\ \hline
        500 & 0.010\%& 0.141\%& 0.003  & 2.477& 2.450\\ \hline
        1000 & 0.006\%& 0.102\%& 0.003  & 2.586& 2.589\\ \hline
        2000 & 0.005\%& 0.125\%& 0.016& 2.602& 2.617\\ \hline
        5000 & 0.006\%& 0.250\%& 0.188& 2.672& 2.859\\\hline
    \end{tabular}    
    \label{tab:3}
\end{table}

\begin{remark}

Algorithm \ref{alg:num} may fail to find a good approximate solution using SDP solvers and 
EM refinement possibly cannot recover the right permutation from the solution extracted from $\xi_{sdp}$, leading to the relative error of the solution returned by algorithm \ref{alg:num} diverging very far away from most of the trials. 
Here we manifest a typical outlier on the condition $n=4, m=500, SNR=60$dB, the median of relative errors of the trials are $err_{EM}=0.016\%$ and $err_{sdp} = 0.396\%$,
the outlier appeared in our experiment is $err_{EM}=165.957\% $ and $err_{sdp} = 167.938\%$. For small $n=4$ and noise $SNR=80$dB or  $SNR=100$dB, the proportion of the outliers is less than  $5\%$. 
The proportion of the outliers is between $5\%$ and $35\%$ for other $n$ and $SNR$, increasing along with the number of variables and noise. 

{\color{black}The outliers are mainly caused by noise interfering SDP algorithm.  The rank sequence of SDP relaxation without noise under the assumption of a unique solution should stabilize at $1$. However, in outliers, the numerical rank sequence increases beyond 1 under a certain numerical tolerance concerning the noise. Set $n=4, m=50, SNR =60$dB and tolerance $0.01$ for the numerical rank.  In $50$ trials, $2$ outliers were recorded; these two outliers have numerical rank sequence $1,3,5,5$, and $1, 2, 3, 3$.  Many factors may lead to the increase of the rank sequence.  We observed that the scale of the
coordinates of the real solution $\xi^*$ varied acutely in two outliers. 
The outlier with numerical rank sequence $1,3, 5,5$ corresponds to the real solution $\xi^* = (0.753,0.081, 0.326, 0.879
)^t$; the outlier with rank sequence $1,2,3,4$ corresponds to the real solution $\xi^* = (0.414, 0.515, 0.004, 0.624)^t$ . }
\end{remark}

Table \ref{tab:3} shows that the accuracy $err_{sdp}$ and $err_{EM}$  of algorithm \ref{alg:num} is stable with different $m$. As $m$ increases, the time $TA3$ consumed mildly rises because the EM refinement part of algorithm \ref{alg:num} needs to sort more entries. However, compared with the sorting procedure, the main computation in algorithm \ref{alg:num} occurs in solving the SDP problem, which is irrelevant to $m$. Therefore, $TA3$ increases rather slowly when $m$ becomes larger. Table \ref{tab:3} also reveals that for $m = 500$ or $1000$, the portion of time consumed on EM refinement is insignificant in the overall time consumed; hence, the overall time of algorithm \ref{alg:num} will be reported in further experiments, and the time of EM refinement will be omitted in tables below.

\subsubsection{numerical algorithm \ref{alg:num}, corrupted measurement $y$, algorithm performance on different magnitude of noise $SNR$, $n = 5, 6$}\label{exp:4}
\hfill

Experiment \ref{exp:4} examines for each $n$ the running time and relative error of algorithm \ref{alg:num} with respect to the magnitude of noise, fixing $m = 500$.

\begin{table}[!ht]
    \centering
    \caption{CPU time, and relative error of $\xi_{EM}$ and $\xi_{sdp}$ for algorithm 3 with different magnitude of noise, $n =3,m=500$ fixed, $20$ trials median.}
    \begin{tabular}{|l|l|l|l|}
    \hline
        $SNR/$dB & $TA3/$s & $ err_{sdp}$ & $ err_{EM}$  \\ \hline
       100 & 0.27  & 1.74E-05 & 1.74E-05  \\ \hline
       80 & 0.29  & 0.029\% & 0.016\%  \\ \hline
       60 & 0.31  & 0.266\% & 0.017\%  \\ \hline
       50 & 0.31  & 0.748\% & 0.058\%  \\ \hline
       40 & 0.31  & 4.508\% & 0.328\%  \\ \hline     
      30 & 0.35  & 10.417\% & 2.447\%  \\ \hline
        
    \end{tabular}
    \label{tab:4}
\end{table}

\begin{table}[!ht]
    \centering
    \caption{CPU time, and relative error of $\xi_{EM}$ and $\xi_{sdp}$ for algorithm 3 with different magnitude of noise, $n =4,m=500$ fixed, $20$ trials median.}
    \begin{tabular}{|l|l|l|l|}
    \hline
        $SNR/$dB & $TA3/$s & $ err_{sdp}$ & $ err_{EM}$  \\ \hline
        100 & 2.85  & 2.14E-05 & 2.14E-05  \\ \hline
         80 & 3.04  & 0.021\% & 0.019\%  \\ \hline
         60 & 2.85  & 0.280\% & 0.021\%  \\ \hline
        50 & 3.02  & 0.603\% & 0.057\%  \\ \hline
        40 & 2.74  & 3.916\% & 0.431\%  \\ \hline
        30 & 2.72  & 11.993\% & 5.170\%  \\ \hline      
    \end{tabular}
    \label{tab:5}
\end{table}

\begin{table}[!ht]
    \centering
    \caption{CPU time, and relative error of $\xi_{EM}$ and $\xi_{sdp}$ for algorithm 3 with different magnitude of noise, $n =5,m=500$ fixed, $20$ trials median.}
    \begin{tabular}{|l|l|l|l|}
    \hline
        $SNR/$dB & $TA3/$s & $ err_{sdp}$ & $ err_{EM}$  \\ \hline
        100 & 6.21 & 0.004\% & 4.17E-05  \\ \hline
         80 & 6.29  & 0.043\% & 0.018\%  \\ \hline
60 & 6.24  & 0.419\% & 0.021\%  \\ \hline
50 & 7.13  & 1.585\% & 0.090\%  \\ \hline
        
        40 & 7.18  & 6.777\% & 0.690\%  \\ \hline

    \end{tabular}
    \label{tab:6}
\end{table}
\end{comment}
\begin{table}[!ht]
    \centering
    \caption{CPU time, and relative error of $\xi_{EM}$ and $\xi_{sdp}$ for algorithm 3 with different magnitude of noise, $n =5,m=500$ fixed, $20$ trials median.}
    \begin{tabular}{|l|l|l|l|}
    \hline
        $SNR/$dB & $TA3/$s & $ err_{sdp}$ & $ err_{EM}$  \\ \hline
        100 & 5.46& 0.003\%& 0.003\%\\ \hline
         80 & 5.58& 0.043\% & 0.021\%\\ \hline
60 & 6.50& 0.220\%& 0.010\%\\ \hline
50 & 6.97& 1.430\%& 0.058\%\\\hline

    \end{tabular}
    \label{tab:6}
\end{table}
\begin{table}[!ht]
    \centering
    \caption{CPU time, and relative error of $\xi_{EM}$ and $\xi_{sdp}$ for algorithm 3 with different magnitude of noise, $n =6,m=500$ fixed, $20$ trials median.}
    \begin{tabular}{|l|l|l|l|}
    \hline
        $SNR/$dB & $TA3/$s & $ err_{sdp}$ & $ err_{EM}$  \\ \hline
        100 & 185.83& 0.002\%& 0.002\%\\ \hline
        80 & 184.59& 0.023\%& 0.019\%\\ \hline
        60 & 253.38& 0.546\%& 0.011\%\\ \hline
        50 & 220.52& 1.224\%& 0.044\%\\ \hline

    \end{tabular}
    \label{tab:7}
\end{table}

Tables 
\ref{tab:6}, \ref{tab:7} show that as the noise builds up, the $ err_{sdp}$ becomes larger, and EM refinement technique successfully improves the precision of the solution in all these experiments. For small noise $SNR = 100$dB or $80$dB, the $ err_{EM}$ is not sensitive to $n$. This is not so strange because the precision of $\xi_{EM}$ depends majorly on the sorting procedure. When the noise is relatively small, with high probability EM technique finds the right permutation $\pi$. Once the permutation $\pi$ is deduced correctly, the precision of classical linear regression only depends on the condition number of $A$ and the relative noise $\lVert\frac{y_c}{A^{*}\xi^{*}}\rVert_{2}$ imposed on the measurement. For general $n = 5,6$, the condition number of $A$ with random entries of standard normal distribution behaves well, and the magnitude of relative noise is less sensitive to $n$. Hence, the relative error after EM refinement $ err_{EM}$ behaves well in small noise circumstances.

\subsubsection{comparison between algorithm AIEM \cite{Tsakiris2020} and algorithm \ref{alg:num}}\label{exp:6}
\hfill
Algorithm AIEM was introduced in \cite{Tsakiris2020}. AIEM uses homotopy method \cite{morgan1987homo} to solve the polynomial system $Q_n$ (\ref{eqn:2}).  {\color{black}Homotopy continuation method solves polynomial systems by transforming the original system into a continuous path from an easily solvable system to the desired one. It introduces an auxiliary parameter $t$ and constructs a family of homotopy equations with $t$, which smoothly deforms from a known system to the target system to be solved.} The homotopy solver 
returns at most $n!$ roots, among which the real solution will be selected by  EM sorting procedure. The relative error for $\xi_{homo}$ the solution from homotopy solver after selection by sorting  is defined as:
$ err_{homo} := \frac{\lVert \xi^{*}-\xi_{homo}\rVert_{2}}{\lVert \xi^{*}\rVert_{2}}.$
\begin{table}[ht]
\begin{center}
    \caption{CPU time and relative error for AIEM and algorithm \ref{alg:num} with different $n$. $m=500,SNR = 80$dB, 20 trials median.}
    \begin{tabular}{|c|c|c|c|c|c|c|}
    \hline   \multicolumn{1}{|c|}{}& \multicolumn{3}{|c|}{AIEM(homotopy)}  & \multicolumn{3}{|c|}{algorithm \ref{alg:num}(SDP)}  \\
    \hline $n$ & $T_{AIEM}/$s & $err_{homo}$ & $err_{EM}$ & $TA3/$s & $err_{sdp}$ & $err_{EM}$ \\
        \hline
        3 & 0.12  & 0.043\% & 0.012\% & 0.29  & 0.029\% & 0.016\%  \\ \hline
        4 & 0.23  & 0.037\% & 0.016\% & 3.04  & 0.021\% & 0.019\%  \\ \hline
        5 & 0.79  & 0.040\% & 0.014\% & 5.58  & 0.043\% & 0.021\%  \\ \hline
        6 & 9.06  & 0.039\% & 0.019\% & 184.80  & 0.053\% & 0.011\%  \\ \hline
    \end{tabular}
    \label{tab:8}
    \end{center}
\end{table}

Table \ref{tab:8} shows that as $n$ increases, AIEM using the homotopy solver HOM4PS2~\cite{homotopy} is more efficient than algorithm  \ref{alg:num} using the SDP solver Sedumi. 
The SDP solver Sedumi used in algorithm \ref{alg:num} implemented the traditional interior-point method in its inner iterations, which is numerically more stable but is relatively slow compared to modern SDP solvers such as SDPNAL+. In \cite{Tsakiris2020}, the homotopy solver Bertini \cite{Bertini} was used to solve the polynomial system $Q_n$, for $n=6$ AIEM needs $2243$ seconds to return the solution. Our experiments suggest that different software and implementations can influence the time and accuracy of AIEM and algorithm \ref{alg:num}, which should be further examined and compared in the future.

{\color{black}
\subsubsection{Computational complexity analysis} \label{sec4.2.6}\hfill
\\

The complexity of solving a zero-dimensional  polynomial system using Groebner bases 
 is single exponential $d^{O(n)}$
\cite{lakshman1991complexity_Gb}, where $d$ is the degree of polynomials and 
$n$ is the number of variables. In the context of unlabeled sensing and algorithm
\ref{alg:symb_n}($d=n$), \ref{alg:symb_extra}($d =n+1$), the overall complexity of symbolic algorithm should be $n^{O(n)} = O(\exp(n\log n))$.

For general input, $\exists n!$ solutions for homotopy continuation algorithm, and 
EM procedure picks one among the $n!$ solutions \cite{Tsakiris2020}. Fixing the precision of the result, the complexity of the homotopy continuation algorithm is proportional to the number of solutions, providing a lower bound $\Omega(n!)$.

For the SDP relaxation algorithm \ref{alg:num}, the interior method costs $O(N\sqrt{M})$, where $N$ is the cost of solving Newton system in iterations, $M$ characterizes the scale of semi-definite constrains \cite{nemirovski2004interior}. Denote $m_s$ as the size of the moment matrix. In the context of unlabeled sensing and algorithm \ref{alg:num}, $N =O(m_s^6)$, $M = O(m_sn)$, and 
$m_s \leq \binom{n+d}{d}=\binom{2n+1}{n}$.

Fixing the precision of results, an upper bound of the complexity of algorithm \ref{alg:num} can be estimated as $O\left(\binom{2n+1}{n}^{6.5}\sqrt{n}\right)$. Comparing  the complexity of homotopy continuation
algorithm and algorithm \ref{alg:num}:
\begin{equation}
    \frac{cost(HC)}{cost(SDP)} = \Omega\left(\frac{n!}{\binom{2n+1}{n}^{6.5}\sqrt{n}}\right) \subseteq \Omega(\exp(n)),
\end{equation}
 we know  that in the perspective of computational complexity, algorithm \ref{alg:num}
 has an exponential speed-up over the homotopy continuation algorithm. 
 
 For the case $n=5$,  the homotopy continuation algorithm in Bertini took half an hour to find  120 solutions \cite{Tsakiris2020}, while the  SDP algorithm in Sedumi used only  $5.58$  seconds to find the unique solution.  In 
our experiment, we used HOM4PS2 to solve the polynomial system, and it took only $0.79$ seconds to find  120 solutions.  Different implementations seem to influence the actual runtime of the algorithms. For low-rank matrix completion problems, there are also other practical algorithms; future research should test other implementations.
 }

\section{Conclusion}\label{conclusion}
In this paper, we give a positive answer to the open problem \ref{openproblem}
by showing the polynomial system $Q_{n+1}$ consisting of the $n+1$ Newton polynomials has a unique solution. The main Theorem \ref{thm:1} is proved by considering the birationality of the algebraic morphisms defined among several algebraic varieties. Since $Q_{n+1}$ has a unique solution, the unlabeled sensing problem can be reduced to the classic rank-one moment matrix completion problem, and the unique solution can be recovered at the lowest relaxation order $\lceil \frac{n+1}{2}\rceil$.

Enlightened by the main Theorem \ref{thm:1} , we develop symbolic and numerical algorithms \ref{alg:symb_n}, \ref{alg:symb_extra} and \ref{alg:num}. Algorithm \ref{alg:symb_n} and \ref{alg:symb_extra} are based on the Groebner basis computation, and algorithm \ref{alg:num} is based on semidefinite programming.

We test and compare the performance of these algorithms for different $n,m$ and $SNR$. Our experiments show that for clean input without noise, algorithm \ref{alg:symb_extra} solving the overdetermined system $Q_{n+1}$ is faster than algorithm \ref{alg:symb_n}, \ref{alg:num} and the existing algorithm AIEM, as the Groebner basis of the ideal generated by polynomials in $Q_{n+1}$ has simple linear form for lexicographic monomial order and graded monomial order. For corrupted input data with noise, our experiment shows that algorithm \ref{alg:num} can return an approximate solution within several minutes.
 The rank sequence of the moment matrices in algorithm \ref{alg:num} stabilizes at $1$ within the numeric tolerance when using clean input data. However, we also notice that due to the degree of the polynomial system $Q_{n+1}$ is relatively high for the moment relaxation method, algorithm \ref{alg:num} may fail to recover the unique solution, which should be improved in the future.

\subsection*{Acknowledgements}\hfill

Lihong  Zhi is  supported by the National Key R$\&$D Program of China (2023YFA1009401) and the National Natural
Science Foundation of China (12071467). Manolis C. Tsakiris is supported by the National Key R$\&$D Program of China (2023YFA1009402). Jingyu Lu would like to thank the help of Tianshi Yu on the experiments. The authors acknowledge the support of the Institut Henri Poincaré (UAR 839 CNRS-Sorbonne Université) and LabEx CARMIN (ANR-10-LABX-59-01).

	\bibliographystyle{abbrv}
	\bibliography{LiangBook, LiangArticle, zhi,polysolvema,exp, Liangzu}

\end{document}